\documentclass[11pt]{amsart}
 \usepackage{graphicx}
\usepackage{xstring}
\usepackage{forloop}
\usepackage{bbm, amsmath, amsfonts, amscd, latexsym, amsthm, amssymb, graphicx, mathrsfs, comment, color, soul}
\usepackage[colorlinks=true]{hyperref} 
\usepackage{abstract}
\usepackage{mathtools}
\usepackage{caption}
\usepackage[dvipsnames]{xcolor}

\usepackage{tikz-cd}
\tikzset{
  symbol/.style={
    draw=none,
    every to/.append style={
      edge node={node [sloped, allow upside down, auto=false]{$#1$}}}
  }
}

\makeatletter
\newif\if@check@engine  \@check@enginetrue 
\makeatother
\usepackage[usenames,dvipsnames]{pstricks}

\usepackage{booktabs}

\newtheorem{theor}{\hspace{1cm}{\sc Theorem}}[section]
\newtheorem{utver}[theor]{\hspace{1cm}{\sc Proposition}}

\newtheorem{lemma}[theor]{\hspace{1cm}{\sc Lemma}}

\newtheorem*{utver*}{\hspace{1cm}{\sc Proposition}}
\theoremstyle{definition}

\newtheorem{defin}[theor]{\hspace{1cm}{\sc Definition}}
\newtheorem{exa}[theor]{\hspace{1cm}{\sc Example}}

\newtheorem{rem}[theor]{\hspace{1cm}{\sc Remark}}

\newtheorem{quest}[theor]{\hspace{1cm}{\sc Question}}

\newcommand{\Vol}{\mathop{\rm Vol}\nolimits}

\newcommand{\id}{\mathop{\rm id}\nolimits}

\newcommand{\rk}{\mathop{\rm rk}\nolimits}

\newcommand{\vol}{\mathop{\rm Vol}\nolimits}
\newcommand{\conv}{\mathop{\rm conv}\nolimits}

\newcommand{\Trop}{\mathop{\rm Trop}\nolimits}

\newcounter{idx}

\newcommand{\rotraise}[1]{
  \StrLen{#1}[\slen]
  \forloop[-1]{idx}{\slen}{\value{idx}>0}{
    \StrChar{#1}{\value{idx}}[\crtLetter]
    \IfSubStr{tlQWERTZUIOPLKJHGFDSAYXCVBNM}{\crtLetter}
      {\raisebox{\depth}{\rotatebox{180}{\crtLetter}}}
      {\raisebox{1ex}{\rotatebox{180}{\crtLetter}}}}
}

\newcommand{\bigslant}[2]{{\raisebox{.2em}{$#1$}\left/\raisebox{-.2em}{$#2$}\right.}}
\newcommand{\area}{\mathop{\rm Area}\nolimits}

\renewcommand{\emph}[1]{{\it {\color{NavyBlue} #1}}}

\def\R{\mathbb R}
\def\N{\mathbb N}
\def\Z{\mathbb Z}

\def\C{\mathbb C}
\def\CC{({\mathbb C}^\star)}

\def\T{\mathbb T}

\begin{document}

\begin{center}{\Large \sc Basecondary polytopes}

\vspace{3ex}

{\sc Alexander Esterov, Arina Voorhaar}
\end{center}

\vspace{3ex}

\section{Introduction}

Many (if not most) of convex polytopes, important for combinatorial and algebraic geometry, are closely related to secondary polytopes of point configurations, or base polytopes of submodular functions, or their numerous variations and generalizations.

The aim of this text is to introduce the class of {\it basecondary polytopes}. This class includes (and allows to study uniformly) the aforementioned ones, as well as some others, e.g. appearing as Newton polytopes of important discriminant hypersurfaces: see Example \ref{exa0} below. 

Most notably, this includes the discriminant of the Lyashko--Looijenga map, which is important for enumerative geometry of ramified coverings and cannot be reduced (by far) to Gelfand--Kapranov--Zelevinsky's A-discriminants and secondary polytopes.

\subsection{Basecondary functions} The basecondary polytope will be defined by its support function from the following data:

-- A ground set $\overline m:=\{1,\ldots,m\}$;

-- A map $A:\overline m\to\R^n$ (whose image in subsequent algebraic geometry applications plays the role of the support of a general Laurent polynomial);

-- A function $F:2^{\overline m}\to\R\cup\{-\infty\}$ (which will have to be nearly submodular, to assure the existence of the basecondary polytope).

\begin{defin}\label{basec}
The {\it basecondary function} $\mathscr{B}_F$ is a piecewise linear function on the space $\R^{\overline m}$. Its value at a 
point $\lambda:\overline m\to\R$ 
is defined as $$\sum_{\gamma}\gamma(0) \cdot (F\{(\lambda-\gamma)\geq 0\}-F\{(\lambda-\gamma)> 0\}) \cdot \Vol A\{(\lambda-\gamma)=\max\}.$$
Here the sum is taken over all affine linear functions $\gamma:\R^n\to\R$ (of which all but finitely many give zero terms), $\Vol$ is the $n$-dimensional lattice volume of the convex hull, and  $\lambda-\gamma:=\lambda-\gamma\circ A$ refers to a function on $\overline m$ under a small abuse of notation.

Whenever the basecondary function is convex, there is a unique polytope with this support function. It is called the {\it basecondary polytope} and is denoted by $\mathcal B_F$.
\end{defin}
Note that the basecondary function by its definition depends only on the values of the function $F$ on sets of more than $n$ elements.
\begin{exa}\label{exa0}
0. If $n=0$ and $F$ is submodular, then $\mathcal B_F$ (for the unique $A:\overline m\to\R^0$) is the base polytope of $F$, modulo the Minkowski summand $F(\overline{m})\cdot($standard simplex$)$. This essentially dates back to Lov\'{a}sz, see Example \ref{exa: dim0_base_polytope} for details.

\vspace{1ex}

1. If $n=1$, and $-F$ is the indicator function of a point, then $\mathcal B_F$ is the (shifted) secondary polytope of the configuration $A(\overline m)\subset\Z$. As observed by Gelfand, Kapranov and Zelevinsky \cite{GKZ}, it is the Newton polytope of the $A$-discriminant, i.e. the resultant of the polynomials $$\sum_{a\in A(\overline m)} c_a x^a \quad\mbox{ and }\quad \sum_{a\in A(\overline m)} a c_a x^a,$$
regarded as a polynomial of the coefficients $c_a$.

\vspace{1ex}

2. More generally, the Newton polytope of the resultant of 
$$\sum_{a\in A(\overline m)} c_a x^a \quad\mbox{ and }\quad \sum_{a\in A(\overline m)} v_a c_a x^a$$
differs from the $A$-secondary polytope by the Minkowski summand $\mathcal B_F$, where $F(I)$ is the rank of the matrix $\begin{pmatrix}1\\v_a\end{pmatrix},\,{a\in I}$. 
This example covers e.g. the resultant of a polynomial and its higher order derivative. Notice that $F$ is the rank function of the respective rank 2 matroid, hence submodular.

\vspace{1ex} 

3. If $n=2$, $A(\overline m)$ is the set of vertices of a convex $m$-gon, and $-F$ is the indicator function of a point, then $\mathcal B_F$ is a (shifted) $m$-associahedron.

\vspace{1ex}

4. If $n=1$ and $F(I)=-\gcd A(I)$ (an exotic nearly submodular function), then $\mathscr B_F$ is closely related to the discriminant of the {\it Lyashko--Looijenga map}, assigning to a rational function the divisor of its critical values. The singularity strata of this map enumerate topological types of rational functions; in particular, the critical locus parameterizes the set of all rational functions with non-generic topology (i.e. {\it non-Morse functions}, having a degenerate critical point or two critical points with the same value). This map is especially important in the enumerative geometry of ramified coverings (see e.g. \cite{lz}). 

Applying the LL map to Laurent polynomials $\sum_{a\in A(\overline m)} c_a x^a$, its critical locus (i.e. the locus of all non-Morse Laurent polynomials) is given by a polynomial equation in the coefficients $c_a$ ({\it the Morse discriminant}). The Newton polytope of this equation is similar in importance to the Newton polytope of the usual discriminant, fundamentally studied starting from \cite{GKZ}.

The support of the dual fan of this polytope can be found from \cite{jems} (without a fan structure). Furthermore, the support function of this polytope was computed in \cite{V23}. The initial motivation for our work is to gain a combinatorial understanding of this polytope. We show that the support function of this polytope, modulo a Minkowski summand equal to an iterated fiber simplex (as in \cite{bs}), is the basecondary function $\mathscr B_{-\gcd}$ (Theorem \ref{theor:np_of_ms} below). The aim of the text is to prove this theorem and the following one.
\end{exa}
For an $(n+2)$-element set $J\subset\bar m$ such that $A(J)$ is not in an affine hyperplane, let $J_0\subset J$ be the minimal subset such that $A(J_0)$ is affine dependent.
\begin{theor}[Proved in Section \ref{ssproof}]\label{mainth}
1. For any $A$, assume that the function $F$ is submodular for sets of size at least $n$, i.e. $F(A\cap B)+F(A\cup B)\leq F(A)+F(B)$ whenever $|A\cap B|\geq n$. Then the basecondary function $\mathscr B_F$ is convex, upon adding a sufficiently large multiple of the support function of the $A(\overline m)$-secondary polytope. 

2. The basecondary function $\mathscr{B}_F$ is itself convex if moreover
\begin{equation}\label{eq:F_conditions}
\sum_{k\in J_0}F\big(J\setminus k\big)-(|J_0|-1)F\big(J\big)-F\big(\overline{m}\big)\geqslant 0 \mbox{ for all } J\subset\bar m,\,|J|=n+2.
\end{equation}
\end{theor}

\begin{rem}
Definition \ref{basec} admits a further mild generalization, covering examples 1-4 for arbitrary $A$ with $n>1$, permutoassociahedra, and many other important polytopes. We shall study this extension ({\it mixed basecondary polytopes}) later.
\end{rem}

\subsection{A tropical interpretation.} 

We recall the tropical interpretation of 
secondary polytopes, and then offer a similar interpretation for the basecondary function $\mathscr B_{-\gcd}$. 

The tropical hyperfield $\T$ is the set $\R\cup\{-\infty\}$ with operations $x\cdot_\T y:=x+y$ and $x+_\T y=\max(x,y)$ for $x\ne y$ or $[-\infty,x]$ otherwise. In particular, we have $0_\T=-\infty$, and the tropical torus  $\R^n=(\T\setminus\{0_\T\})^n$ is naturally identified with the real part of the Lie algebra of $\CC^n$. In this torus, a tropical Laurent polynomial $\mathfrak f$ of $n$ variables defines the zero locus $\{{\mathfrak x}\in\R^n\,|\, {\mathfrak f}({\mathfrak x})\ni 0_\T\}$ (shortcut to $\{{\mathfrak f}=0_\T\}$), which is a codimension 1 polyhedral complex, see e.g. \cite{maclagansturmfels} or \cite{mikhalkinrau}.

A complex Laurent polynomial $f(x)=\sum_a c_a x^a$ defines a tropical polynomial $\Trop f({\mathfrak x})=\sum_{a\,|\,c_a\ne 0} {\mathfrak x}^a$. The tropical fan of the hypersurface $\{f=0\}\subset\CC^n$ is a codimension 1 polyhedral cone complex $\Trop\{f=0\}:=\{\Trop f=0_\T\}$. Combinatorially, it is the dual fan to the Newton polytope of $f$, and, geometrically, the Hausdorff limit of $($amoeba of $f=0)/t$ as $t\to\infty$.

A tropical correspondence theorem is an observation that a certain question about polynomials has the same answer over a certain field (e.g. $\C$ or $\R$) and the hyperfield $\T$. 

One of the first tropical correspondence theorems can be seen behind a classical Gelfand--Kapranov--Zelevinsky result \cite{GKZ}: the Newton polytope of the $A$-discriminant is a Minkowskii summand of the $A$-secondary polytope $S_A$.

Indeed, the $A$-discriminant is the closure of all 
polynomials $f\in\C^{A(\bar m)}:=\{\sum_{k\in \bar m}c_k x^{A(k)}\}$ having a degenerate root $x\in\CC^n$ (i.e. such that $f(x)=0$ and $df(x)=0$). If it is given by one equation $D_A=0$ (which happens for most of $A$), then its tropical fan is the dual fan to the Newton polytope of $D_A$.

A point ${\mathfrak x}\in\R^n$ is a degenerate root of a tropical polynomial $\mathfrak f$, if the maximal value of its terms at ${\mathfrak x}$ is contributed by more than two terms. Defining the {\it tropical discriminant} as the set of all tropical polynomials in $\R^{A(\bar m)}$ with a degenerate root, we notice that it is a tropical hypersurface, given by the tropical polynomial equation ${\mathfrak D}_A=0$ with unit coefficients and the Newton polytope $S_A$ (the secondary polytope).

Thus, the GKZ theorem implies that the tropical fan of the usual discriminant is contained in the topical discriminant (and actually coincides for $n=1$): $\Trop \{D_A=0\} \subset \{{\mathfrak D}_A=0\}$. Reverting the reasoning, we can actually see this inclusion as a justification of why we define a degenerate root of a tropical polynomial as above.

\begin{rem}\label{mainrem}
One may have an impression that the above notion of the degenerate root of a tropical polynomial is just artificially adjusted to imply this ``correspondence theorem''. Actually, it has a more fundamental motivation (though can be indeed ``predicted'' from this correspondence theorem): defining the tropicalization of a family of complex polynomials $f_t(x)=\sum c_a(t) x^a,\,c_a\in\C(t)$, as $\Trop f:=\sum \deg c_a {\mathfrak x}^a$, a tropical polynomial has a degenerate root if and only if it is the tropicalization of a family of complex polynomials having a degenerate root.
\end{rem}

We now give a similar tropical interpretation to the observation of Example \ref{exa0}.4 about the basecondary function $\mathscr B_{-\gcd}$ and the Morse discriminant. Again, it reduced to a suitable definition a Morse tropical polynomial (for $n=1$).

\begin{defin}\label{deftropmorse}
1. A critical point of a tropical polynomial $\mathfrak f$ is its root, and a critical value is the maximal value of $\mathfrak f$ at this point (the minimal one being $0_\T$ by the definition of a root).

2. The critical point ${\mathfrak x}$ is said to be degenerate, if $\mathfrak f$ has two pairs of terms taking equal values at ${\mathfrak x}$ (the pairs are allowed to overlap but not coincide; note that one such pair of terms is assured by the fact that ${\mathfrak x}$ is critical).

3. A tropical Laurent polynomial is called Morse, if all of its critical points are nondegenerate, and no critical values coincide.

4. The tropical Morse discriminant is the codimension 1 polyhedral complex in $\R^{A(\bar m)}$ that consists of all non-Morse tropical polynomials.
\end{defin}

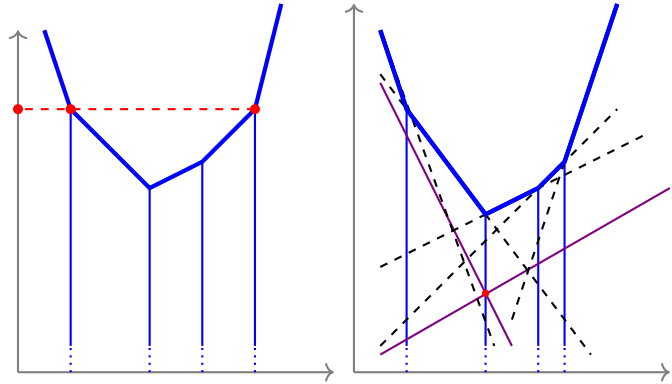
\begin{figure}[ht]
\begin{tikzpicture}[scale=0.35]
			\filldraw[color=white, fill=white, thick](-3,3) circle (0.1);
			\draw[thick, gray,->] (0,0)--(12,0);
			\draw[thick, gray,->] (0,0)--(0,13);
			\draw[ultra thick, blue] (1,13)--(2,10)--(5,7)--(7,8)--(9,10)--(10,14);
			\draw[thick, blue] (2,10)--(2,1);
			\draw[thick, blue] (5,7)--(5,1);
			\draw[thick, blue] (7,8)--(7,1);
			\draw[thick, blue] (9,10)--(9,1);
			\draw[thick, blue,dotted] (2,0)--(2,1);
			\draw[thick, blue,dotted] (5,0)--(5,1);
			\draw[thick, blue,dotted] (7,0)--(7,1);
			\draw[thick, blue,dotted] (9,0)--(9,1);
			\draw[thick, red, dashed] (9,10)--(0,10);
			\filldraw[color=red, fill=red, thick](9,10) circle (0.15);
			\filldraw[color=red, fill=red, thick](2,10) circle (0.15);
			\filldraw[color=red, fill=red, thick](0,10) circle (0.15);
			
		\end{tikzpicture}
		\begin{tikzpicture}[scale=0.35]
			\draw[thick, gray,->] (0,0)--(12,0);
			\draw[thick, gray,->] (0,0)--(0,14);
			\draw[ultra thick, blue] (1,13)--(2,10)--(5,6)--(7,7)--(8,8)--(10,14);
			\draw[thick, blue] (2,10)--(2,1);
			\draw[thick, blue] (5,6)--(5,1);
			\draw[thick, blue] (7,7)--(7,1);
			\draw[thick, blue] (8,8)--(8,1);
			\draw[thick, blue,dotted] (2,0)--(2,1);
			\draw[thick, blue,dotted] (5,0)--(5,1);
			\draw[thick, blue,dotted] (7,0)--(7,1);
			\draw[thick, blue,dotted] (8,0)--(8,1);
			\draw[thick, violet] (1,11)--(6,1);
			\draw[thick, violet] (1,0.67)--(12,7);
			\filldraw[color=red, fill=red, thick](5,3) circle (0.1);
			\draw[thick,dashed] (1,13)--(5.33,1);
			\draw[thick,dashed] (1,11.33)--(9,0.67);
			\draw[thick,dashed] (1,1)--(10,10);
			\draw[thick,dashed] (6,2)--(10,14);
			\draw[thick, dashed] (1,4)--(11,9);
			\draw[ultra thick, blue] (1,13)--(2,10)--(5,6)--(7,7)--(8,8)--(10,14);
			
		\end{tikzpicture}
  \caption{A tropical polynomial with coinciding critical values (on the left) and a tropical polynomial with a degenerate critical point (on the right).}
\end{figure}

Our Example \ref{exa0}.4 now implies that the tropical fan of the Morse discriminant is contained in the tropical Morse discriminant.

Furthermore, Definition \ref{deftropmorse} interprets the codimension $k$ skeleton $S_k$ of the corner locus of the basecondary function $\mathscr B_{-\gcd}$ as the set of tropical Laurent polynomials with a sufficiently degenerate ramification profile, i.e. as the tropical analogue of the singularity strata of the LL map. 

\begin{quest}\label{q1}
To what extent tropical fans of codimension $k$ singularity strata of the LL map belong to the fans $S_k$?
\end{quest}
Our result implies the positive answer for $k=1$, and thus gives a hope for an interesting answer for higher $k$. This would be parallel to the following fact about the Gelfand--Kapranov--Zelevinsky setting, specializing to Remark \ref{mainrem} for $k=1$. Denote by $d$ the number of non-zero roots of a typical polynomial in $\C^{A(\bar m)}$, i.e. $\max {A(\bar m)}-\min {A(\bar m)}$.
\begin{utver}\label{tropseveri1}
If the set $\{$polynomials with $d-k$ roots$\}$ has codimension $k$ (i.e. the expected one) in $\C^{A(\bar m)}$, then its tropical fan is contained in the codimension $k$ skeleton of the secondary fan (i.e. the fan dual to $S_A$).
\end{utver}
This in particular holds for the space of all polynomials of degree $d$ (i.e. when $A=\id:\{1,\ldots,m\}\to\{1,\ldots,m\}$).

\begin{rem} 1. The question about critical points of polynomials is infinitely more complicated than the proposition about roots of polynomials, because we study (polynomial) maps with prescribed singularities over distinct points of the range, vs prescribed singularities over one point 0.

2. Similarly, our result of Example \ref{exa0}.4 is drastically more complicated than the Gelfand--Kapranov--Zelevinsky observation, because we have to study maps with multiple singularities over some point of the range (more specifically, two singularities of type $A_1$), vs one singularity.

3. Question \ref{q1} may relate the basecondary function $\mathscr B_{-\gcd}$ and its tropical interpretation to enumerative geometry of ramified coverings. A well known tropical point of view on this topic (\cite{hannah}, \cite{bbm}) operates with tropical ramified coverings understood as line projections of abstract tropical curves, i.e. weighted graphs without any natural embedding into an ambient space (in contrast to the graph of a tropical polynomial and its projection to the vertical axis as in Figure 1). As a consequence, the conditions like Definition \ref{deftropmorse}.3 and pictures like Figure 1.2 make no sense from the usual point of view. It would be interesting to understand whether they allow to interestingly refine tropical counts of ramified coverings.
\end{rem} 
\subsection*{Acknowledgements}
The second author was funded by Horizon Europe ERC (Grant number: 101045750, Project acronym: HodgeGeoComb).

\section{The support function of a basecondary polytope}\label{sec:basecondary}
\subsection{Preliminaries}\label{sec:preliminaries}

\begin{defin}\label{def:submodular_fun}
Let $S$ be a finite set, $|S|=m.$ A function $F\colon 2^S\to\R$ is called {\it submodular} if, for any subset $X\subset S$ and any elements $x_1,x_2\in S\setminus X$ such that $x_1\neq x_2,$ the following inequality holds: 
\begin{equation}\label{eq:submodular}
F(X\cup\{x_1\})+F(X\cup\{x_2\})\geqslant F(X)+F(X\cup\{x_1,x_2\}).
\end{equation}
\end{defin}

In what follows we will only consider submodular functions with an extra requirement: $F(\emptyset)=0.$ 

\begin{exa}\label{exa:submodular_fn}
The following functions are submodular: 
\begin{itemize}
\item For any $X\subset S,~C(X)=-\dfrac{|X|}{|S|};$
\item the rank function $\rk_M$ on a matroid $M;$
\item for any $X\subset S,~F(S) = \sum\limits_{i=1}^{|X|}(|S|-i+1).$
\end{itemize}
\end{exa}

\begin{defin}\label{def:submodular_poly}
Let $F\colon 2^S\to \R$ be a submodular function such that $F(\emptyset)=0.$ Its {\it submodular polyhedron} is the polyhedron $P_F\subset\R^S$ defined as follows: 
$$P_{F}=\{y\in\R^S\mid \sum_{s\in X}y_s\leqslant F(X) \text{ for all } X\subset S\}.$$
\end{defin}

\begin{defin}\label{def:base_poly}
Let $F\colon 2^S\to \R$ be a submodular function such that $F(\emptyset)=0.$ The polytope $$B_F= P_F\cap \{\sum_{s\in S} y_s=F(S)\}\subset\R^S$$ is called the {\it base polytope} of the function $F.$
\end{defin}\label{lovaszdef}
We now describe the support function of a base polytope.
\begin{defin}\label{def:Lovasz_ext}
Given a function $F\colon \{0,1\}^S\to\R$ with $F(\emptyset)=0,$ we construct a piecewise linear  function $\widetilde{F}\colon\R^S\to\R$ as follows. Let $\sigma\colon S\to S$ be a permutation sorting the coordinates of $x\in\R^S$ in descending order: 
$$x_{s_{\sigma(1)}}\geqslant x_{s_{\sigma(2)}}\geqslant\ldots\geqslant  x_{s_{\sigma(m)}}.$$
Then we set $Y_0=\emptyset,~Y_i=\{\sigma(1),\ldots,\sigma(i)\},~1\leqslant i\leqslant m,$ and define 
\begin{equation}\label{eq:Lovasz_ext}
\widetilde{F}(x)=\sum_{i=1}^{m} x_{s_{\sigma(i)}}(F(Y_i)-F(Y_{i-1})).
\end{equation}
The function $\widetilde{F}$ is called the {\it Lov\'{a}sz extension} of the function $F.$ 
\end{defin}

Below we provide a list of basic properties satisfied by the Lov\'{a}sz extension all of which can be verified by a straightforward computation. 

For a subset $X\subset S,$ by $1_X$ we denote the vector in $\{0,1\}^S$ whose entries are $1$ or $0$ depending on whether the corresponding element of the set $S$ belongs to $X$ or not:
$$(1_X)_s=\begin{cases*}
	1,\text{ if }s\in X\\
	0,\text{ otherwise}.
\end{cases*}$$

\begin{utver}\label{prop:Lovasz_ext}
The Lov\'{a}sz extension $\widetilde{F}$ of a function $F\colon \{0,1\}^S\to\R,$ satisfies the following properties:
\begin{itemize}
	\item For any $X\subset S,$ we have $\widetilde{F}(1_X)=F(X);$
    \item The Lov\'{a}sz extension is continuous;
	\item The Lov\'{a}sz extension is positively homogeneous: $\widetilde{F}(\lambda x)=\lambda\widetilde{F}(x),$ for all $\lambda\geqslant 0.$
\end{itemize}
\end{utver}

The following result concerns convexity of Lov\'{a}sz extension. 

\begin{theor}[\cite{lov} and \cite{edm} respectively]\label{theor:Lovasz_ext_convex}
The Lov\'{a}sz extension $\widetilde{F}$ is convex if and only if the initial function $F\colon \{0,1\}^S\to\R$ is submodular. Moreover, for a submodular function $F,$ its Lov\'{a}sz extension $\widetilde{F}$ is exactly the support function of the base polytope $B_F$ introduced in Definition \ref{def:base_poly}.
\end{theor}

\subsection{Dramatis Person\ae}\label{subsec:notation}
\begin{itemize}
\item[--] a pair $m>1$ and $n\geqslant 0$ of integers;
\item[--] the set $\overline{m}=\{1,\ldots,m\}$ of all integers between $1$ and $m\in\N;$
\item[--] an map $A\colon\overline{m}\to\R^n,$ which is additionally assumed to be injective when $n\geqslant 1;$
\item[--] a function $F\colon \{I\subset  2^{\overline{m}}\mid |I|\geqslant n\}\to\R;$ 
\item[--] a map $\gamma\colon\overline{m}\to\R;$ 
\item[--] a linear map $L\colon \R^n\to\R;$
\item[--] the permutation $\gamma^L\colon\overline{m}\to\overline{m}$ that orders the elements by the values of $\gamma-L\circ A$ in descending order (for the details, see Definitions \ref{def:perm_generic} and \ref{def:perm_circuit} below);
\item[--] the composition $\widetilde{\gamma}^L=(A,\gamma)\circ\gamma^L\colon\overline{m}\to \R^{n+1
}.$
\end{itemize}

In what follows, by $[b_1,\ldots,b_k]$ we will denote ordered tuples of elements $b_i\in\overline{m}.$ Finally, we will use triangular brackets to denote oriented volumes of simplices in $\R^n$ or $\R^{n+1}$ (depending on the context).

\subsection{\texorpdfstring{The basecondary function $\mathscr{B}_F$}{the basecondary function} is piecewise linear} \label{subsec:generic}

\begin{defin}\label{def:simp_L}
A map $L\colon\R^n\to\R$ is called {\it $\gamma$-simplicial} if the function $\gamma-L\circ A$ achieves its maximal value at exactly $n+1$ elements $\{b_1,\ldots,b_{n+1}\}\subset\{1\ldots m\}$ such that the images $A(b_i),~1\leqslant i\leqslant n+1,$ form an affinely independent set in $\R^n.$
\end{defin}

\begin{defin}\label{def:gen_L}
A $\gamma$-simplicial $L$ is called {\it generic} if the values of the function $\gamma-L\circ A$ at the elements of $\{1\ldots m\}\setminus\{b_1,\ldots,b_{n+1}\}$ are all distinct. 
\end{defin}

\begin{defin}\label{def:gen_gamma}
The map $\gamma\colon\{1,\ldots,m\}\to\R$ is called {\it generic} if every $\gamma$-simplicial $L$ is generic. 
\end{defin}

\begin{exa}\label{exa:dim1_generic}
Fix $m=4$ and $n=1.$ Let $A$ be the map sending $[1,2,3,4]$ to $[1,3,6,7]\subset\R.$ Take $\gamma_1\colon[1,2,3,4]\to\R$ such that $\gamma_1(1)=2,~\gamma_1(2)=4,~\gamma_1(3)=5,~\gamma_1(4)=3.$ A linear function $x\mapsto \alpha\cdot x$ is $\gamma_1$-simplicial for $\alpha=\frac{\gamma_1(2)-\gamma_1(1)}{a_2-a_1}=1, \alpha=\frac{\gamma_1(3)-\gamma_1(2)}{a_3-a_2}=\dfrac{1}{3}$ and $\alpha=\frac{\gamma_1(4)-\gamma_1(3)}{a_4-a_3}=-2.$ Moreover, one can easily check that for each of the abovementioned values of $\alpha,$ the function $L$ is generic (see Definition \ref{def:gen_L} and Figure \ref{fig:dim1_generic}). For example, we have $\gamma_1(1)-1\cdot a_1=2=\gamma_1(2)-1\cdot a_2,~\gamma_1(3)-1\cdot a_3=-1,~\gamma_1(4)-1\cdot a_4=-4.$ Therefore, by Definition \ref{def:gen_gamma}, $\gamma_1$ itself is generic. 
\end{exa}

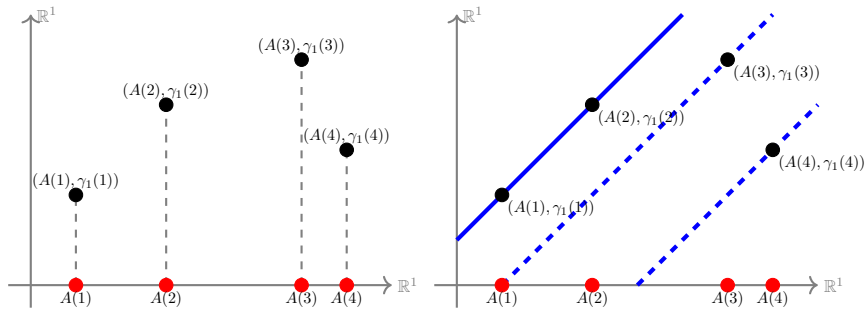
\begin{figure}[ht]
\begin{center}
\begin{tikzpicture}[scale=0.6]
\draw[thick, gray,->] (-0.5,0)--(8,0);
\draw[thick, gray,->] (0,-0.5)--(0,6);
\draw[thick, gray, dashed] (1,0)--(1,2);
\draw[thick, gray, dashed] (3,0)--(3,4);
\draw[thick, gray, dashed] (6,0)--(6,5);
\draw[thick, gray, dashed] (7,0)--(7,3);
\draw[fill,red] (1,0) circle (1.5mm);
\draw[fill,red] (3,0) circle (1.5mm);
\draw[fill,red] (6,0) circle (1.5mm);
\draw[fill,red] (7,0) circle (1.5mm);
\draw[fill,black] (1,2) circle (1.5mm);
\draw[fill,black] (3,4) circle (1.5mm);
\draw[fill,black] (6,5) circle (1.5mm);
\draw[fill,black] (7,3) circle (1.5mm);
\node[right,gray,scale=0.6] at (8,0) {$\R^1$};
\node[right,gray,scale=0.6] at (0,6) {$\R^1$};
\node[below,scale=0.55] at (1,0) {$A(1)$};
\node[below,scale=0.55] at (3,0) {$A(2)$};
\node[below,scale=0.55] at (6,0) {$A(3)$};
\node[below,scale=0.55] at (7,0) {$A(4)$};
\node[above,scale=0.55] at (1,2) {$(A(1),\gamma_1(1))$};
\node[above,scale=0.55] at (3,4) {$(A(2),\gamma_1(2))$};
\node[above,scale=0.55] at (6,5) {$(A(3),\gamma_1(3))$};
\node[above,scale=0.55] at (7,3) {$(A(4),\gamma_1(4))$};
\end{tikzpicture}
\begin{tikzpicture}[scale=0.6]
\draw[thick, gray,->] (-0.5,0)--(8,0);
\draw[thick, gray,->] (0,-0.5)--(0,6);
\draw[blue,ultra thick] (0,1)--(5,6);
\draw[blue,ultra thick, dashed, opacity=0.75] (1,0)--(7,6);
\draw[blue,ultra thick, dashed, opacity=0.5] (4,0)--(8,4);
\draw[fill,red] (1,0) circle (1.5mm);
\draw[fill,red] (3,0) circle (1.5mm);
\draw[fill,red] (6,0) circle (1.5mm);
\draw[fill,red] (7,0) circle (1.5mm);
\draw[fill,black] (1,2) circle (1.5mm);
\draw[fill,black] (3,4) circle (1.5mm);
\draw[fill,black] (6,5) circle (1.5mm);
\draw[fill,black] (7,3) circle (1.5mm);
\node[right,gray,scale=0.6] at (8,0) {$\R^1$};
\node[right,gray,scale=0.6] at (0,6) {$\R^1$};
\node[below,scale=0.55] at (1,0) {$A(1)$};
\node[below,scale=0.55] at (3,0) {$A(2)$};
\node[below,scale=0.55] at (6,0) {$A(3)$};
\node[below,scale=0.55] at (7,0) {$A(4)$};
\node[below right,scale=0.55] at (1,2) {$(A(1),\gamma_1(1))$};
\node[below right,scale=0.55] at (3,4) {$(A(2),\gamma_1(2))$};
\node[below right,scale=0.55] at (6,5) {$(A(3),\gamma_1(3))$};
\node[below right,scale=0.55] at (7,3) {$(A(4),\gamma_1(4))$};
\end{tikzpicture}
\end{center}
\caption{The map $L\colon x\mapsto 1\cdot x$ is a generic $\gamma_1$-simplicial map.}\label{fig:dim1_generic}
\end{figure} 

\begin{exa}\label{exa:dim1_ng}
In the setting of Example \ref{exa:dim1_generic}, take $\gamma_2$ such that $\gamma_2(1)=1,~\gamma_2(2)=3,~\gamma_2(3)=3,~\gamma_2(4)=1.$ Then the function $x\mapsto 0\cdot x$ is $\gamma$-simplicial, but not generic. Therefore, $\gamma_2$ is not generic either.
\end{exa}

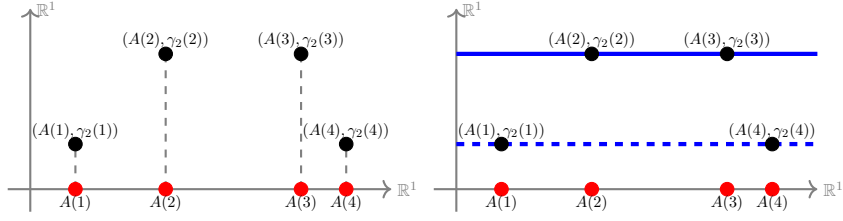
\begin{figure}[ht]
\begin{center}
\begin{tikzpicture}[scale=0.6]
\draw[thick, gray,->] (-0.5,0)--(8,0);
\draw[thick, gray,->] (0,-0.5)--(0,4);
\draw[thick, gray, dashed] (1,0)--(1,1);
\draw[thick, gray, dashed] (3,0)--(3,3);
\draw[thick, gray, dashed] (6,0)--(6,3);
\draw[thick, gray, dashed] (7,0)--(7,1);
\draw[fill,red] (1,0) circle (1.5mm);
\draw[fill,red] (3,0) circle (1.5mm);
\draw[fill,red] (6,0) circle (1.5mm);
\draw[fill,red] (7,0) circle (1.5mm);
\draw[fill,black] (1,1) circle (1.5mm);
\draw[fill,black] (3,3) circle (1.5mm);
\draw[fill,black] (6,3) circle (1.5mm);
\draw[fill,black] (7,1) circle (1.5mm);
\node[right,gray,scale=0.6] at (8,0) {$\R^1$};
\node[right,gray,scale=0.6] at (0,4) {$\R^1$};
\node[below,scale=0.55] at (1,0) {$A(1)$};
\node[below,scale=0.55] at (3,0) {$A(2)$};
\node[below,scale=0.55] at (6,0) {$A(3)$};
\node[below,scale=0.55] at (7,0) {$A(4)$};
\node[above,scale=0.55] at (1,1) {$(A(1),\gamma_2(1))$};
\node[above,scale=0.55] at (3,3) {$(A(2),\gamma_2(2))$};
\node[above,scale=0.55] at (6,3) {$(A(3),\gamma_2(3))$};
\node[above,scale=0.55] at (7,1) {$(A(4),\gamma_2(4))$};
\end{tikzpicture}
\begin{tikzpicture}[scale=0.6]
\draw[thick, gray,->] (-0.5,0)--(8,0);
\draw[thick, gray,->] (0,-0.5)--(0,4);
\draw[blue,ultra thick] (0,3)--(8,3);
\draw[blue,ultra thick, dashed, opacity=0.75] (0,1)--(8,1);
\draw[fill,red] (1,0) circle (1.5mm);
\draw[fill,red] (3,0) circle (1.5mm);
\draw[fill,red] (6,0) circle (1.5mm);
\draw[fill,red] (7,0) circle (1.5mm);
\draw[fill,black] (1,1) circle (1.5mm);
\draw[fill,black] (3,3) circle (1.5mm);
\draw[fill,black] (6,3) circle (1.5mm);
\draw[fill,black] (7,1) circle (1.5mm);
\node[right,gray,scale=0.6] at (8,0) {$\R^1$};
\node[right,gray,scale=0.6] at (0,4) {$\R^1$};
\node[below,scale=0.55] at (1,0) {$A(1)$};
\node[below,scale=0.55] at (3,0) {$A(2)$};
\node[below,scale=0.55] at (6,0) {$A(3)$};
\node[below,scale=0.55] at (7,0) {$A(4)$};
\node[above,scale=0.55] at (1,1) {$(A(1),\gamma_2(1))$};
\node[above,scale=0.55] at (3,3) {$(A(2),\gamma_2(2))$};
\node[above,scale=0.55] at (6,3) {$(A(3),\gamma_2(3))$};
\node[above,scale=0.55] at (7,1) {$(A(4),\gamma_2(4))$};
\end{tikzpicture}
\end{center}
\caption{The map $L\colon x\mapsto 0\cdot x$ is $\gamma_2$-simplicial, but not generic.}\label{fig:dim1_ng}
\end{figure}

\begin{defin}\label{def:perm_generic}
For the given generic $\gamma$ and $\gamma$-simplicial $L$, the permutation \linebreak $\gamma^L\colon\{1,\ldots,m\}\to\{1,\ldots,m\}$ is well-defined (up to even permutations of the first $n+1$ points) by the following two conditions: 
\begin{itemize}
    \item For every $n+1<i\leqslant m,$ the permutation $\gamma^L$ sends $i$ to the element $j,$ at which the function $\gamma-L\circ A$ achieves its $(i-n)-$th highest value;
    \item the oriented volume $\Big\langle A\gamma^L[1,\ldots,n+1]\Big\rangle$ is positive.
\end{itemize}
\end{defin}

\begin{exa}\label{exa:dim0_generic}
If $n=0,$ then there is only one map $A\colon\overline{m}\to\R^0$ and only one linear map $L\colon\R^0\to\R^0.$ The map $L$ is $\gamma$-simplicial if and only if $\gamma$ attains its maximum at exactly one element of $\overline{m}.$ The map $L$ (and thus the map $\gamma$ itself) is generic, if all the values $\gamma(i),~i\in\overline{m}$ are distinct. If $\gamma$ is generic, then the permutation $\sigma\coloneqq\gamma^L$ ordering the elements of $\overline{m}$ by the values of $\gamma$ is  uniquely defined by the first property in Definition \ref{def:perm_generic}. 
\end{exa}

\begin{exa}\label{exa:dim1_permutations}
In Example \ref{exa:dim1_generic}, for a generic map $\gamma_1$ we found all the $\gamma_1$-simplicial (and therefore generic) functions $L_1(x)=x,~L_2(x)=x/3,$ and $L_3=-2x.$ Now let us construct the corresponding permutations $\gamma^{L_1},~\gamma^{L_2},~\gamma^{L_3}$ using Definition \ref{def:perm_generic}. We already computed that $\gamma_1(1)-L_1\circ A(1)=\gamma_1(2)-L_1\circ A(2)>\gamma_1(3)-L_1\circ A(3)>\gamma_1(4)-L_1\circ A(4).$ Thus $\gamma^L(3)=3, \gamma^L(4)=4$ by the first condition, and $\gamma^L(1)=1, \gamma^L(2)=2$ by the second one. Thus we obtain that $\gamma^{L_1}$ is the trivial permutation $[1,2,3,4]\mapsto[1,2,3,4]$. Similarly, since $\gamma_1(2)-L_2\circ A(2)=\gamma_1(3)-L_2\circ A(3)>\gamma_1(1)-L_2\circ A(1)>\gamma_1(4)-L_2\circ A(4),$ we have $\gamma_1^{L_2}\colon[1,2,3,4]\mapsto[2,3,1,4].$ Finally, we have $\gamma_1^{L_3}\colon[1,2,3,4]\mapsto[3,4,2,1].$
\end{exa}

\begin{rem}\label{rem:sum_simplicial}
For any generic $\gamma,$ the basecondary function $\mathscr{B}_F$ from Definition \ref{basec} is linear in the neighborhood of $\gamma$ and can be rewritten as follows: 
\begin{equation}\label{eq:sum_simplicial}
    \mathscr{B}_F(\gamma)=\sum_{\substack{L\in(\R^n)^* \\ L\textrm{ simplicial}}}\sum_{i=1}^m\Big\langle \widetilde{\gamma}^L[1,\ldots,n+1,i]\Big\rangle\Big(F\big(\{1,\ldots,i-1\}\big)-F\big(\{1,\ldots,i\}\big)\Big).
\end{equation}    
\end{rem}

\begin{exa}\label{exa: dim0_base_polytope}
In the setting of Example \ref{exa:dim0_generic}, let us have a more detailed look at the sum (\ref{eq:sum_simplicial}) for $n=0.$ The simplices $\widetilde{\gamma}^L[1,\ldots,n+1,i]$ are just intervals of length $\max_{b\in\overline{m}}(\gamma(b))-\gamma(i)$ and, for generic $\gamma,$ the  expression from Definition \ref{basec} can be rewritten as follows: 
\begin{multline}
 \mathscr{B}_F(\gamma)=\sum_{i=1}^m\Big(\max\limits_{b\in\overline{m}}(\gamma(b))-\gamma(\sigma(i))\Big)\Big(F\big(\{\sigma(1),\ldots,\sigma(i-1)\}\big)-F\big(\{\sigma(1),\ldots,\sigma(i)\}\big)\Big)=\\=\sum\limits_{i\in\overline{m}}\Big(\gamma(\sigma(i))-\gamma(\sigma(1))\Big)\Big(F\big(\{\sigma(1),\ldots,\sigma(i)\}\big)-F\big(\{\sigma(1),\ldots,\sigma(i-1)\}\big)\Big)=\\=\sum\limits_{i\in\overline{m}}\gamma(\sigma(i))\Big(F\big(\{\sigma(1),\ldots,\sigma(i)\}\big)-F\big(\{\sigma(1),\ldots,\sigma(i-1)\}\big)\Big)-\gamma(\sigma(1))\Big (F(\overline{m})-F(\varnothing)\Big).
\end{multline}
Comparing the first summand with (\ref{eq:Lovasz_ext}), we can conclude that it is equal to the Lov\'{a}sz extension of the function $F.$ Therefore, by Theorem \ref{theor:Lovasz_ext_convex}, it is convex if and only if $F$ is submodular (see Definition \ref{def:submodular_fun}) and is equal to the support function of the base polytope of $F$ (see Definition \ref{def:base_poly}). The second summand equals $-F(\overline{m})$ times the support function of the $(m-1)-$simplex.
\end{exa}

\begin{exa}\label{exa:dim1_basecondary}
For the map $\gamma_1$ from Examples \ref{exa:dim1_generic} and \ref{exa:dim1_permutations}, the sum (\ref{eq:sum_simplicial}) can be rewritten as follows: 
\begin{multline}
\mathscr{B}_F(\gamma_1)=\Big\langle \big(a_1,\gamma_1(1)\big),\big(a_2,\gamma_1(2)\big),\big(a_3,\gamma_1(3)\big)\Big\rangle\Big(F\big(\{1,2\}\big)-F\big(\{1,2,3\}\big)\Big)+\\+\Big\langle \big(a_1,\gamma_1(1)\big),\big(a_2,\gamma_1(2)\big),\big(a_4,\gamma_1(4)\big)\Big\rangle\Big(F\big(\{1,2,3\}\big)-F\big(\{1,2,3,4\}\big)\Big)+\\+\Big\langle \big(a_2,\gamma_1(2)\big),\big(a_3,\gamma_1(3)\big),\big(a_1,\gamma_1(1)\big)\Big\rangle\Big(F\big(\{2,3\}\big)-F\big(\{1,2,3\}\big)\Big)+\\+\Big\langle \big(a_2,\gamma_1(2)\big),\big(a_3,\gamma_1(3)\big),\big(a_4,\gamma_1(4)\big)\Big\rangle\Big(F\big(\{1,2,3\}\big)-F\big(\{1,2,3,4\}\big)\Big)+\\+\Big\langle \big(a_3,\gamma_1(3)\big),\big(a_4,\gamma_1(4)\big),\big(a_2,\gamma_1(2)\big)\Big\rangle\Big(F\big(\{3,4\}\big)-F\big(\{2,3,4\}\big)\Big)+\\+\Big\langle \big(a_3,\gamma_1(3)\big),\big(a_4,\gamma_1(4)\big),\big(a_1,\gamma_1(1)\big)\Big\rangle\Big(F\big(\{2,3,4\}\big)-F\big(\{1,2,3,4\}\big)\Big).
\end{multline}
\end{exa}

\subsection{Convexity of \texorpdfstring{the function $\mathscr{B}_F$}{the basecondary function}} \label{subsec:circuital}

\begin{defin}\label{def:circ_L}
A map $L\colon\R^n\to\R$ is called \it{$\gamma$-circuital} if the function $\gamma-L\circ A$ achieves its maximal value at $n+2$ elements $\{b_1,\ldots,b_{n+2}\}\subset\{1\ldots m\}$ such that the set $A(\{b_1,\ldots,b_{n+2})$ is not contained in any $(n-1)-$plane.
\end{defin}

\begin{exa}\label{exa:circuits_dim1}
In the setting of Examples \ref{exa:dim1_generic} and \ref{exa:dim1_ng}, let us consider the map \linebreak $\gamma_3\colon[1,2,3,4]\to \R$ such that $\gamma_3(1)=\gamma_3(2)=\gamma_3(3)=3$ and $\gamma_3(4)=1.$ Then the function $L\colon x\mapsto \R$ is $\gamma_3$-circuital.
\end{exa}

\begin{figure}[ht]
\begin{center}
\begin{tikzpicture}[scale=0.6]
\draw[thick, gray,->] (-0.5,0)--(8,0);
\draw[thick, gray,->] (0,-0.5)--(0,4);
\draw[thick, gray, dashed] (1,0)--(1,3);
\draw[thick, gray, dashed] (3,0)--(3,3);
\draw[thick, gray, dashed] (6,0)--(6,3);
\draw[thick, gray, dashed] (7,0)--(7,1);
\draw[fill,red] (1,0) circle (1.5mm);
\draw[fill,red] (3,0) circle (1.5mm);
\draw[fill,red] (6,0) circle (1.5mm);
\draw[fill,red] (7,0) circle (1.5mm);
\draw[fill,black] (1,3) circle (1.5mm);
\draw[fill,black] (3,3) circle (1.5mm);
\draw[fill,black] (6,3) circle (1.5mm);
\draw[fill,black] (7,1) circle (1.5mm);
\node[right,gray,scale=0.6] at (8,0) {$\R^1$};
\node[right,gray,scale=0.6] at (0,4) {$\R^1$};
\node[below,scale=0.55] at (1,0) {$A(1)$};
\node[below,scale=0.55] at (3,0) {$A(2)$};
\node[below,scale=0.55] at (6,0) {$A(3)$};
\node[below,scale=0.55] at (7,0) {$A(4)$};
\node[above,scale=0.55] at (1,3) {$(A(1),\gamma_3(1))$};
\node[above,scale=0.55] at (3.3,3) {$(A(2),\gamma_3(2))$};
\node[above,scale=0.55] at (6,3) {$(A(3),\gamma_3(3))$};
\node[above,scale=0.55] at (7,1) {$(A(4),\gamma_3(4))$};
\end{tikzpicture}
\begin{tikzpicture}[scale=0.6]
\draw[thick, gray,->] (-0.5,0)--(8,0);
\draw[thick, gray,->] (0,-0.5)--(0,4);
\draw[blue,ultra thick] (0,3)--(8,3);
\draw[blue,ultra thick, dashed, opacity=0.75] (0,1)--(8,1);
\draw[fill,red] (1,0) circle (1.5mm);
\draw[fill,red] (3,0) circle (1.5mm);
\draw[fill,red] (6,0) circle (1.5mm);
\draw[fill,red] (7,0) circle (1.5mm);
\draw[fill,black] (1,3) circle (1.5mm);
\draw[fill,black] (3,3) circle (1.5mm);
\draw[fill,black] (6,3) circle (1.5mm);
\draw[fill,black] (7,1) circle (1.5mm);
\node[right,gray,scale=0.6] at (8,0) {$\R^1$};
\node[right,gray,scale=0.6] at (0,4) {$\R^1$};
\node[below,scale=0.55] at (1,0) {$A(1)$};
\node[below,scale=0.55] at (3,0) {$A(2)$};
\node[below,scale=0.55] at (6,0) {$A(3)$};
\node[below,scale=0.55] at (7,0) {$A(4)$};
\node[above,scale=0.55] at (1,3) {$(A(1),\gamma_3(1))$};
\node[below,scale=0.55] at (3,3) {$(A(2),\gamma_3(2))$};
\node[above,scale=0.55] at (6,3) {$(A(3),\gamma_3(3))$};
\node[above,scale=0.55] at (7,1) {$(A(4),\gamma_3(4))$};
\end{tikzpicture}
\end{center}
\caption{The map $L\colon x\mapsto 0\cdot x$ is $\gamma_3$-circuital.}\label{fig:dim1_circ}
\end{figure}
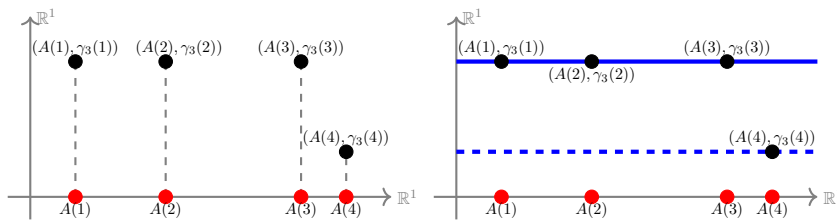

\begin{rem}\label{rem:pq}
If the set $A(\{b_1,\ldots,b_{n+2}\})$ is not contained in hyperplane, then upon a suitable reordering of points, there exists exactly one affine relation of the form $\sum_{i=1}^p \lambda_iA(b_i)=\sum_{j=1}^{q}\mu_jA(b_j)$ with positive $\lambda_i,~1\leqslant i\leqslant p,$ and $\mu_j,~1\leqslant j\leqslant q.$ Note that the relation does not necessarily have to involve all $(n+2)$ points (see Example \ref{exa:circuits_dim2}), thus $p+q$ can be strictly smaller than $n+2.$
\end{rem}

\begin{exa}\label{exa:circuits_dim2}
Take $n=2.$ In the setting of Remark \ref{rem:pq}, Figure \ref{fig:dim2_circ} shows all possible configurations of the points $A(b_1),A(b_2),A(b_3),A(b_4),$ where $\{b_1,b_2,b_3,b_4\}$ is the set of maximizers of the function $\gamma-L\circ A$ for a $\gamma$-circuital map $L\colon\overline{m}\to\R^2.$ For the first configuration, we have $p=1$ and $q=3,$ while for the second one, we have $p=2$ and $q=2$. Finally, for the last configuration, we have $p=1$ and $q=2$, thus, as mentioned in Remark \ref{rem:pq}, we have a strict inequality $p+q<n+2.$
\end{exa}
\begin{figure}[ht]
\begin{center}
\begin{tikzpicture}
\draw[gray,thick] (-0.5,0)--(2.5,0);
\draw[gray,thick] (-0.25,-0.5)--(1.25,2.5);
\draw[gray,thick] (2.25,-0.5)--(0.75,2.5);
\draw[fill,black] (0,0) circle (1mm);
\draw[fill,black] (1,0.75) circle (1mm);
\draw[fill,black] (1,2) circle (1mm);
\draw[fill,black] (2,0) circle (1mm);
\end{tikzpicture}
\begin{tikzpicture}
\draw[fill,white] (-2,0) circle (1mm);
\draw[gray,thick] (0,-0.5)--(0,2.5);
\draw[gray,thick] (-0.5,0)--(2.5,0);
\draw[gray,thick] (2,-0.5)--(2,2.5);
\draw[gray,thick] (-0.5,2)--(2.5,2);
\draw[fill,black] (0,0) circle (1mm);
\draw[fill,black] (2,0) circle (1mm);
\draw[fill,black] (0,2) circle (1mm);
\draw[fill,black] (2,2) circle (1mm);
\end{tikzpicture}
\begin{tikzpicture}
\draw[gray,thick] (-0.5,0)--(2.5,0);
\draw[gray,thick] (-0.25,-0.5)--(1.25,2.5);
\draw[gray,thick] (2.25,-0.5)--(0.75,2.5);
\draw[fill,white] (-2,0) circle (1mm);
\draw[fill,black] (0,0) circle (1mm);
\draw[fill,black] (1,0) circle (1mm);
\draw[fill,black] (2,0) circle (1mm);
\draw[fill,black] (1,2) circle (1mm);
\end{tikzpicture}
\end{center}
\caption{Circuits in dimension 2.}\label{fig:dim2_circ} 
\end{figure}
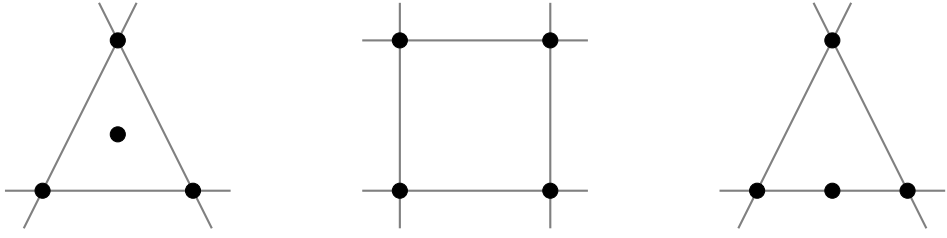

\begin{defin}\label{def:perm_circuit}
For the given generic $\gamma$ and a $\gamma$-circuital map $L$, the permutation $\gamma^L\colon\{1,\ldots,m\}\to\{1,\ldots,m\}$ is defined by the following conditions: 
\begin{itemize}
    \item For every $n+2<i\leqslant m,$ the permutation $\gamma^L$ sends $i$ to the element $j,$ at which the function $\gamma-L\circ A$ achieves its $(i-n-1)-$th highest value;
    \item $\conv(A\circ\gamma^L[1,\ldots,p])$ and $\conv(A\circ\gamma^L[p+1,\ldots,p+q])$ (see Remark \ref{rem:pq}) are simplices of complementary dimensions intersecting at the common interior point; 
    \item in the notation of Remark \ref{rem:pq}, 
    \begin{multline}\label{eq:perm_circuit}
      \vol(A\gamma^L[1,\ldots,n+2]) = \sum_{i=1}^p (-1)^i \Big\langle A\gamma^L\big([1,\ldots,n+2] \setminus i\big) \Big\rangle = \\=\sum_{i=p+1}^{p+q} (-1)^{i+1} \Big\langle  A\gamma^L\big([1,\ldots,n+2] \setminus i\big)\Big\rangle.
    \end{multline}
\end{itemize}
\end{defin}
\begin{rem}
The permutation from Definition \ref{def:perm_circuit} is not uniquely defined, but this will not affect the correctness of subsequent definitions involving $\gamma^L.$     
\end{rem}
\begin{exa}\label{exa:perm_circuital}
Let us revisit Example \ref{exa:circuits_dim2} and use Definition $\ref{def:perm_circuit}$ and label the points in the corresponding configurations (see Figure \ref{fig:dim2_circ_perm} below) so that the conditions 2 and 3 from Definition \ref{def:perm_circuit} hold. The labellings below obviously satisfy condition 2, let us check whether condition 3 also holds. Indeed, for the first configuration, we have 
\begin{multline*}
\vol(A\gamma^L[1,2,3,4])=(-1)^1\Big\langle A\gamma^L[2,3,4]\Big\rangle=\\=(-1)^{2+1}\Big\langle A\gamma^L[1,3,4]\Big\rangle+(-1)^{3+1}\Big\langle A\gamma^L[1,2,4]\Big\rangle+(-1)^{4+1}\Big\langle A\gamma^L[1,2,3]\Big\rangle.    
\end{multline*}
Similarly, for the second configuration, we have 
\begin{multline*}
\vol(A\gamma^L[1,2,3,4])=(-1)^1\Big\langle A\gamma^L[2,3,4]\Big\rangle=(-1)^2\Big\langle A\gamma^L[1,3,4]\Big\rangle=\\=(-1)^{3+1}\Big\langle A\gamma^L[1,2,4]\Big\rangle+(-1)^{4+1}\Big\langle A\gamma^L[1,2,3]\Big\rangle. 
\end{multline*}
Finally, the following equality holds for configuration 3:
$$\vol(A\gamma^L[1,2,3,4])=(-1)^1\Big\langle A\gamma^L[2,3,4]\Big\rangle=(-1)^{2+1}\Big\langle A\gamma^L[1,3,4]\Big\rangle+(-1)^{3+1}\Big\langle A\gamma^L[1,2,4]\Big\rangle.$$ 
\end{exa}
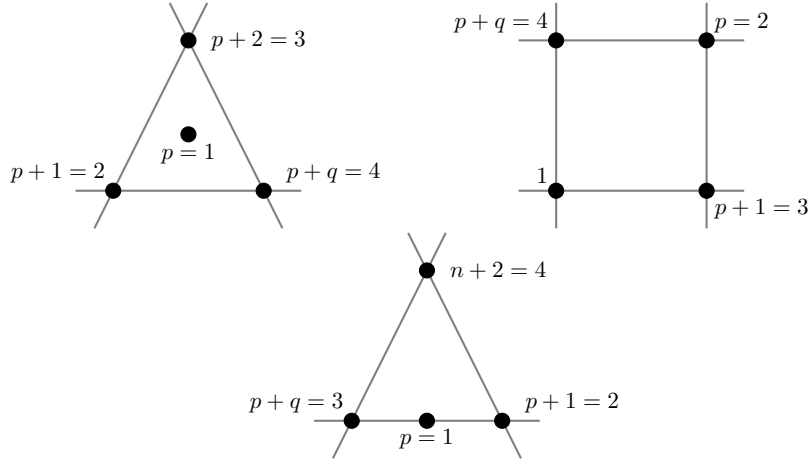
\begin{figure}[ht]
\begin{center}
\begin{tikzpicture}
\draw[gray,thick] (-0.5,0)--(2.5,0);
\draw[gray,thick] (-0.25,-0.5)--(1.25,2.5);
\draw[gray,thick] (2.25,-0.5)--(0.75,2.5);
\draw[fill,black] (0,0) circle (1mm);
\draw[fill,black] (1,0.75) circle (1mm);
\draw[fill,black] (1,2) circle (1mm);
\draw[fill,black] (2,0) circle (1mm);
\node[below,scale=0.8] at (1,0.73) {$p=1$};
\node[above left,scale=0.8] at (0,0) {$p+1=2$};
\node[right,scale=0.8] at (1.2,2) {$p+2=3$};
\node[above right,scale=0.8] at (2.2,0) {$p+q=4$};
\end{tikzpicture}
\begin{tikzpicture}
\draw[fill,white] (-2,0) circle (1mm);
\draw[gray,thick] (0,-0.5)--(0,2.5);
\draw[gray,thick] (-0.5,0)--(2.5,0);
\draw[gray,thick] (2,-0.5)--(2,2.5);
\draw[gray,thick] (-0.5,2)--(2.5,2);
\draw[fill,black] (0,0) circle (1mm);
\draw[fill,black] (2,0) circle (1mm);
\draw[fill,black] (0,2) circle (1mm);
\draw[fill,black] (2,2) circle (1mm);
\node[above left,scale=0.8] at (0,0) {$1$};
\node[above right,scale=0.8] at (2,2) {$p=2$};
\node[above left,scale=0.8] at (0,2) {$p+q=4$};
\node[below right,scale=0.8] at (2,0) {$p+1=3$};
\end{tikzpicture}
\begin{tikzpicture}
\draw[gray,thick] (-0.5,0)--(2.5,0);
\draw[gray,thick] (-0.25,-0.5)--(1.25,2.5);
\draw[gray,thick] (2.25,-0.5)--(0.75,2.5);
\draw[fill,white] (-2,0) circle (1mm);
\draw[fill,black] (0,0) circle (1mm);
\draw[fill,black] (1,0) circle (1mm);
\draw[fill,black] (2,0) circle (1mm);
\draw[fill,black] (1,2) circle (1mm);
\node[below,scale=0.8] at (1,0) {$p=1$};
\node[above left,scale=0.8] at (0,0) {$p+q=3$};
\node[right,scale=0.8] at (1.2,2) {$n+2=4$};
\node[above right,scale=0.8] at (2.2,0) {$p+1=2$};
\end{tikzpicture}
\end{center}
\caption{Labelling the points of the circuits in dimension 2.}\label{fig:dim2_circ_perm}
\end{figure}
The following can be seen e.g. from \cite{GKZ}.
\begin{lemma}\label{lemma:A}
In the same notation as above, if $\gamma$ is a generic point of the corner locus of the support function of the secondary polytope, then 

\begin{itemize}
    \item There is exactly one subset $I\subset\overline{m}$ such that $(A,\gamma)(I)$ is a circuit, $\conv((A,\gamma)(I))$ is a face of $\conv((A,\gamma)(\overline{m}))$, and all $\gamma$-simplicial and $\gamma$-circuital maps $L$ are generic;
    \item in a neighborhood of $\gamma$ the function $\mathscr{B}_F$ is a tropical binomial. 
\end{itemize} 
\end{lemma}

Consider a linear family $\gamma_t$ such that $\gamma_0$ is as described in Lemma \ref{lemma:A} and for any $t\neq 0,~\gamma_t$ is generic (see Definition \ref{def:gen_gamma}).

\begin{lemma}\label{lemma:B}
Let $L$ be $\gamma$-circuital with $\gamma=\gamma_0,$ and let the simplex $\widetilde{\gamma}^{L}[1,\ldots,p]$ be a face of $\conv((A,\gamma_t)(\overline{m}))$ for some negative $t.$ Then the following equality holds for every $j>n+2:$ 
\begin{multline}\label{eq:volumes}
(-1)^{n+1}\Big\langle\widetilde\gamma_t^L [1,\ldots,n+2] \Big\rangle=\vol\Big(\conv((A,\gamma_t)[1,\ldots,n+2])\Big)=\\=\sum_{\substack{L_+\in(\R^n)^*\\ \textrm{is }\gamma_{-t}\textrm{-simplicial}}}\sum_{i=1}^p (-1)^{i+1}\Big\langle(A,\gamma_t)\circ\gamma_{-t}^{L_+}\big([1,\ldots,n+2,j]\setminus i\big)\Big\rangle+\\+
\sum_{\substack{L_-\in(\R^n)^*\\ \textrm{is }\gamma_{t}\textrm{-simplicial}}}\sum_{i=p+1}^{p+q} (-1)^{i+1} \Big\langle\widetilde\gamma_t^{L_-}\big([1,\ldots,n+2,j]\setminus i\big)\Big\rangle.     
\end{multline}
\end{lemma}

Let $\gamma$ and $I$ be as in Lemma \ref{lemma:A}, and let $L$ be one of the $\gamma$-circuital maps (see Definition \ref{def:circ_L}). Then the set $I$ is contained in the maximum locus of the function $\gamma-L\circ A.$ 

Without loss of generality, assume that $I=\{1,\ldots,p+q\}\subset\overline{m}.$ Denote all $\gamma$-circuital maps by $L_{\alpha},~\alpha\in\Lambda,$ the maximum loci of $\gamma-L_{\alpha}\circ A$ by $I_\alpha\supset I,$ respectively. Note that by Definition \ref{def:circ_L}, each of the sets $I_\alpha$ has cardinality $n+2.$

\subsection{The proof of Theorem \ref{mainth}}\label{ssproof} In what follows, we will deduce the sufficient conditions on the function $F$ for the basecondary function $\mathscr{B}_F$ to be convex on the intersections of adjacent full-dimensional cones of the secondary fan. 

Let $C_-$ and $C_+$ be the full-dimensional cones of the $A$-secondary fan containing $\gamma_t$ and $\gamma_{-t}$ respectively (and thus we have $\gamma_0\in C_-\cap C_{+}$). 

We need to make sure that the difference $\mathscr{B}_F\mid_{C_-}(\gamma_t)-\mathscr{B}_F\mid_{C_+}(\gamma_{t})$ for $-1\ll t<0,$ (where $\mathscr{B}_F\mid_{C_+}(\gamma_{t})$ is understood as the linear continuation from small positive $t$) is non-positive. 

Plugging the expression from Remark {\ref{rem:sum_simplicial}}, we observe that all terms cancel, except for those for $L=L_\alpha$. Collecting the surviving terms, we denote them by
$$\mathscr{B}_F\mid_{C_-}(\gamma_t)-\mathscr{B}_F\mid_{C_+}(\gamma_{t})=\sum_{\alpha\in\Lambda}\mathcal C_\alpha,
$$
that is, there will be a summand $\mathcal C_{\alpha}$ for every $\gamma$-circuital $L_{\alpha},~\alpha\in\Lambda.$ To guarantee that the sum $\sum_{\alpha\in\Lambda}\mathcal C_{\alpha}$ is non-positive, it is sufficient to make sure each of the summands $\mathcal C_{\alpha}$ is non-positive.

Now fix an arbitrary $\alpha\in\Lambda.$ Let $L_{\alpha}$ and $I_{\alpha}\subset\overline{m}$ be the corresponding  $\gamma$-circuital map and the maximum locus, respectively. Without loss of generality, we can assume that $I_{\alpha}=[1,\ldots,n+2]\subset\overline{m},$ and the corresponding permutation $\gamma^{L_{\alpha}}$ sends each $j\in I_{\alpha}$ to itself. 

\begin{rem}
Let $L_-$ be a generic $\gamma_t$-simplicial map for some $1\ll t<0.$ Then the function $\gamma_t-L_-\circ A$ attains its maximal value at $I_{\alpha}\setminus k$ for some $p+1\leqslant k\leqslant p+q.$ The corresponding permutation $\gamma_t^{L_-}\colon\overline{m}\to\overline{m}$ maps the set $[1,\ldots,n-1]$ to $I_{\alpha}\setminus k,$ moreover, $\gamma_t^{L_-}(n+2)=k$ and for $j>n+2,$ we have $\gamma_t^{L_-}(j)=\gamma^{L_{\alpha}}(j).$ 

A similar observation is true for a generic $\gamma_{-t}$-simplicial map $L_+.$ The function $\gamma_{-t}-L_+\circ A$ attains its maximal value at $I_{\alpha}\setminus k$ for some $1\leqslant k\leqslant p.$ The corresponding permutation $\gamma_{-t}^{L_+}\colon\overline{m}\to\overline{m}$ maps the set $[1,\ldots,n-1]$ to $I_{\alpha}\setminus k,$ moreover, $\gamma_{-t}^{L_+}(n+2)=k$ and for $j>n+2,$ we have $\gamma_{-t}^{L_+}(j)=\gamma^{L_{\alpha}}(j).$ 
\end{rem}

Then for any $-1\ll t<0,$ we have:

\begin{multline*}
\mathcal C_{\alpha}=\sum_{\substack{L_-\in(\R^n)^*\\ \textrm{is }\gamma_{t}\textrm{-simplicial}}}\sum_{k=p+1}^{p+q}\Bigg[(-1)^{k+1}\Big\langle\widetilde{\gamma}_t^{L_-}[1,\ldots,\hat{k},\ldots,n+2,k]\Big\rangle\Big(F\big(I_{\alpha}\setminus k\big)-F\big(I_{\alpha}\big)\Big)+\\+\sum_{j=n+3}^{m}(-1)^{k+1}\Big\langle\widetilde{\gamma}_t^{L_-}[1,\ldots,\hat{k},\ldots,n+2,j]\Big\rangle\cdot\\ \cdot\Big(F\big(I_{\alpha}\cup\gamma^{L_{\alpha}}([n+3,\ldots,j-1])\big)-F\big(I_{\alpha}\cup\gamma^{L_{\alpha}}([n+3,\ldots,j])\big)\Big)\Bigg]-\\-\sum_{\substack{L_+\in(\R^n)^*\\ \textrm{is }\gamma_{-t}\textrm{-simplicial}}}\sum_{k=1}^{p}\Bigg[(-1)^{k}\Big\langle(A,\gamma_t)\circ \gamma_{-t}^{L_+}[1,\ldots,\hat{k},\ldots,n+2,k]\Big\rangle\Big(F\big(I_{\alpha}\setminus k\big)-F\big(I_{\alpha}\big)\Big)+\\+\sum_{j=n+3}^{m}(-1)^{k}\Big\langle(A,\gamma_t)\circ \gamma_{-t}^{L_+}[1,\ldots,\hat{k},\ldots,n+2,j]\Big\rangle\cdot\\ \cdot\Big(F\big(I_{\alpha}\cup\gamma^{L_{\alpha}}([n+3,\ldots,j-1])\big)-F\big(I_{\alpha}\cup\gamma^{L_{\alpha}}([n+3,\ldots,j])\big)\Big)\Bigg].
\end{multline*}

Using formula (\ref{eq:perm_circuit}) from Definition \ref{def:perm_circuit} and Lemma \ref{lemma:B}, we can simplify the expression above as follows: 

\begin{multline*}
\mathcal C_{\alpha}=\sum_{k=1}^{p+q}\Bigg[(-1)^{k+1+n+2-k}\Big\langle\widetilde{\gamma}_t^{L_{\alpha}}[1,\ldots,n+2]\Big\rangle\Big(F\big(I_{\alpha}\setminus k\big)-F\big(I_{\alpha}\big)\Big)\Bigg]+\\+(-1)^{n+1}\Big\langle\widetilde\gamma_t^{L_{\alpha}} [1,\ldots,n+2] \Big\rangle\Big(F\big(I_{\alpha}\big)-F\big(\overline{m}\big)\Big)=\\=(-1)^{n+1}\Big\langle\widetilde\gamma_t^{L_{\alpha}} [1,\ldots,n+2] \Big\rangle\Big(\sum_{k=1}^{p+q}F\big(I_{\alpha}\setminus k\big)-(p+q-1)F\big(I_{\alpha}\big)-F\big(\overline{m}\big)\Big)=\\=\vol\Big(\conv((A,\gamma_t)[1,\ldots,n+2])\Big)\Big(\sum_{k=1}^{p+q}F\big(I_{\alpha}\setminus k\big)-(p+q-1)F\big(I_{\alpha}\big)-F\big(\overline{m}\big)\Big)=\\=\vol\Big(\conv((A,\gamma_t)I_{\alpha})\Big)\Big(\sum_{k\in I}F\big(I_{\alpha}\setminus k\big)-(p+q-1)F\big(I_{\alpha}\big)-F\big(\overline{m}\big)\Big).
\end{multline*}

{\it Proof of Theorem \ref{mainth}.1.}
It is enough to prove that $\mathscr{B}_F$ is convex in the interior of every full dimensional cone of the $A$-secondary fan. Let $C_\mathcal{T}$ be such a cone, corresponding to the triangulation $\mathcal{T}$ of the configuration $A(\bar m)$ with the maximal simplices $A(T),\, T\in \mathcal{T}$.
Denote 

-- the Lov\'{a}sz function (Definition \ref{lovaszdef}) for $F$ (arbitrarily extended to a submodular function on the whole $2^{\bar m}$) by ${\mathcal F}:\R^m\to\R$; 

-- the linear map sending $\gamma\in\R^m$ to the tuple of volumes of $(n+1)$-simplices $\bigl(A,\gamma\bigr)(T\cup\{k\})\subset\R^{n+1},\,k\in\bar m$, by $\bar T:\R^m\to\R^m$.

The restriction of $\mathscr{B}_F$ to the cone $C_\mathcal{T}$ equals the convex function $\sum_{T\in\mathcal{T}} {\mathcal F}\circ\bar T$. $\hfill\square$

{\it Proof of Part 2.} It is enough to prove that $\mathscr{B}_F$ is convex outside a codimension 2 polyhedral complex. We shall prove this for the codimension 2 skeleton of the $A$-secondary fan. 

This amounts to verifying convexity of $\mathscr{B}_F$ in the relative interior of codimension 1 and full dimensional secondary cones. For codimension 1 cones, the convexity is proved by the reasoning of Subsection \ref{ssproof} above. For full dimensional cones, see the proof of Part 1.  $\hfill\square$

\section{The Newton polytope of the Morse Discriminant revisited}\label{sec:Morse_disc}
The rest of the paper is devoted to representing the Newton polytope $\mathcal M_A$ of the Morse Discriminant as the Minkowski sum of an iterated fiber simplex, the basecondary polytope and a multiple of the secondary polytope. This section is organized as follows. We will first reformulate the main result of \cite{V23}, which is a formula for the support function $\mu_A$ of the polytope $\mathcal M_A,$ in the notation of Section \ref{sec:basecondary}. Then we will represent one if its summands as an iterated fiber simplex.

\subsection{The support function of the polytope \texorpdfstring{$\mathcal M_A$}{}}
We fix a set $A=\{a_1<\ldots<a_m\}\subset \Z\setminus\{0\}$ affinely generating $\Z$ and by $\C^A$ denote the space of univariate Laurent polynomials with support $A.$

\begin{defin}\label{def:caustic}
The {\it caustic} in $\C^A$ is the set of all $f\in\C^A$ such that the map $f\colon\CC\to\C$ has a degenerate critical point. 
\end{defin}

\begin{defin}\label{def:ms}
The {\it Maxwell stratum}  in $\C^A$ is the set of all $f\in\C^A$ such that the map $f\colon\CC\to\C$ has a pair of coinciding critical values taken at distinct points.
\end{defin}

\begin{defin}\label{def:Morse_pol}
A polynomial $f\in\C^A$ is called {\it Morse} if it belongs neither to the caustic nor to the Maxwell stratum.
\end{defin}

\begin{defin}\label{def:Morse_disc}
The {\it Morse discriminant} is the closure of the set of all non-Morse polynomials $f\in\C^A.$ It is given by the polynomial $h_m^2h_c,$ where $h_m$ and $h_c$ are polynomials defining the Maxwell stratum and the caustic respectively if these two sets are hypersurfaces. Otherwise we set the corresponding defining polynomial to $1$. 
\end{defin}

Before we write down the formula for the support function $\mu_A$ of the Newton polytope $\mathcal M_A$ of the Morse discriminant, let us introduce a bit of notation. 

A covector $\gamma\in(\R^A)^*$ can be viewed as a function $\gamma\colon A\to\R.$ 
\begin{defin}\label{N_and_delta}
Let $\gamma\in(\R^A)^*$ be a covector with non-negative coordinates. 
With $\gamma$ we associate the following polytopes: 

$$N_{\gamma}=\conv(\{(a,0)\mid a\in A\}\cup\{(a,\gamma(a))\mid a\in A\})\subset\R\langle e_1,e_3\rangle,$$
and 
$$\overline{\Delta}_{\gamma}=\conv(\{(a,0,0)\mid a\in A\}\cup\{(a,0,\gamma(a))\mid a\in A\}\cup\{0,1,0\})\subset\R\langle e_1,e_2,e_3\rangle.$$
\end{defin}

\begin{defin}\label{P_gamma}
For a given covector $\gamma\in(\R^A)^*$  with non-negative coordinates, the polygon $\overline{P}_{\gamma}\subset\R^2$ is defined as the Minkowski integral $\int \pi_1|_{\overline{\Delta}_{\gamma}}$ with respect to the projection $\pi_1\colon\R^3\twoheadrightarrow\bigslant{\R^3}{\langle e_2,e_3\rangle}.$
\end{defin}

In the notation of Section \ref{sec:basecondary}, let $A\colon\overline m\to\Z$ be the map sending $i\in \overline{m}$ to $a_i\in \Z,$ and $F\colon 2^{\overline{m}}\to\R$ be the function mapping $\{b_1,\ldots,b_k\}$ to $-\gcd(A(b_1),\ldots,A(b_k))$ and the empty set to $0$. Finally, a function $\gamma\colon\{a_1,\ldots,a_m\}\to\R$ can be viewed as a composition of a map $\gamma\colon\overline{m}\to\R$ and the map $A\colon\overline m\to\Z.$
Then Theorem 3.14 of \cite{V23} can be rewritten as follows. 

\begin{theor}[\cite{V23}]\label{theor:Morse_support}
Up to a shift, the support function of the Newton polytope of the Morse discriminant can be computed via the following formula: 
\begin{equation}\label{eq:Morse_support}
\mu_A(\gamma)=\area(\overline{P}_{\gamma})+\mathscr B_F(\gamma)-3\area(N_{\gamma}).
\end{equation}
\end{theor}

The second and the third summands of (\ref{eq:Morse_support}) are the basecondary function $\mathscr B_{-\gcd}$ and a multiple of the support function of the secondary polytope $\mathcal S_A.$ We will now discuss the first summand of (\ref{eq:Morse_support}).
\subsection{The first summand}
\begin{defin}
A polytope $Q\subset\R^{n}$ is called {\it homogeneous}, if it lies in an affine hyperplane perpendicular to the vector $(1,\ldots,1).$
\end{defin}

Below we will construct the polytope in $\R^A$ which is uniquely defined by the following two properties: it is homogeneous, and its support function restricted onto the set of covectors in $(\R^A)^*$ with non-negative coordinates coincides with the first summand in formula (\ref{eq:Morse_support}). 

Choose a basis $(e_1,e_2,\alpha_1,\ldots,\alpha_n)$ in $\Z^{n+2}$ and consider the simplex $$\Omega_A=\conv(\{a_i\cdot e_1+\alpha_i\mid 1\leqslant i\leqslant n\}\cup\{e_2\})\subset\R^{n+2}.$$

Equivalently, the simplex $\Omega_A\subset\R^{n+2}$ can be defined as the Newton polytope of the polynomial $c_0+\sum_{a\in A} c_a x^a\in\C[x,c_0; c_a\mid a\in A].$ 

By $\mathcal S_{A\cup 0}$ we will denote its Minkowski integral $\int \rho_1|_{\Omega_A}\subset\R^{n+1}$ with respect to the projection $\rho_1\colon\R^{n+2}\twoheadrightarrow\bigslant{\R^{n+2}}{\langle e_2,\alpha_1,\ldots,\alpha_n\rangle}$ (\cite{bs0}). It is well-known that this polytope is equal to the secondary polytope for the set $A\cup\{0\}.$

Taking the Minkowski integral $\int \rho_2|_{\mathcal S_{A\cup 0}}$ of the polytope $\mathcal S_{A\cup 0}$ with respect to the projection $\rho_2\colon\R^{n+1}\twoheadrightarrow\bigslant{\R^{n+1}}{\langle \alpha_1,\ldots,\alpha_n\rangle}$ we obtain the polytope $\Sigma_A\subset\R^A.$

\begin{lemma}\label{lemma:iterated}
In the same notation as above, the polytope $\Sigma_A\subset\R^A$ is the homogeneous polytope whose support function restricted to the set of covectors in $(\R^A)^*$ with non-negative coordinates coincides with the first summand in formula (\ref{eq:Morse_support}).
\end{lemma}

Before we prove Lemma \ref{lemma:iterated}, let us introduce a bit more notation. 

By $\overline{\Omega}_A\subset\R^{n+2}$ we denote the Newton polytope of the polynomial $c_0+\sum_{a\in A} c_a x^a+x^{a_1}+x^{a_n}\in\C[x,c_0; c_a\mid a\in A].$ It is clear from the definition that we have the inclusion $\Omega_A\subset\overline{\Omega}_A.$ 

By $\overline{\mathcal S}_{A\cup 0}$ we denote the Minkowski integral $\int \rho_1|_{\overline{\Omega}_A},$ and by $\overline{\Sigma}_{A}$ we denote the polytope $\int \rho_2|_{\overline{\mathcal S}_{A\cup 0}}.$ We also have the inclusions $\mathcal S_{A\cup 0}\subset \overline{\mathcal S}_{A\cup 0}$ and $\Sigma_{A} \subset\overline{\Sigma}_{A}.$ 

With a given covector $\gamma\in(\R^n)^*$ having non-negative entries we associate the projection $\varphi_{\gamma}\colon \R^{n+2}\to\R^3=\R\langle e_1,e_2,e_3\rangle$ defined as follows: $e_1\mapsto e_1, e_2\mapsto e_2, \alpha_i\mapsto \gamma_i\cdot e_3,~1\leqslant i\leqslant n.$

\begin{utver}\label{prop:non_negative}
The support function of the polytope $\overline{\Sigma}_A$ restricted to the set of covectors in $(\R^A)^*$ with non-negative coordinates coincides with the first summand in formula (\ref{eq:Morse_support}). 
\end{utver}

\begin{proof}
First, let us note that for any covector $\gamma$ with non-negative coordinates, the image of $\overline{\Omega}_A$ under the projection $\varphi_{\gamma}$ is a pyramid $\overline{\Delta}_{\gamma}$ from Definition \ref{N_and_delta}. 

The first summand in formula (\ref{eq:Morse_support}) is the area of the fiber polygon $\overline{P}_{\gamma}=\int\pi_1|_{\overline{\Delta}_{\gamma}}$ of $\overline{\Delta}_{\gamma}$ with respect to the projection $\pi_1\colon\R^3\twoheadrightarrow\bigslant{\R^3}{\langle e_2,e_3\rangle}$. 

Allowing ourselves a slight abuse of notation, we denote the projection  $\R^{n+1}\twoheadrightarrow \R^2,$ defined by $e_2\mapsto e_2, \alpha_i\mapsto \gamma_i\cdot e_3,~1\leqslant i\leqslant n,$ by the same symbol $\varphi_{\gamma}$. 

Since taking the Minkowski integral with respect to $\rho_1$ (or $\pi_1$) commutes with the projection $\varphi_{\gamma}$, we can represent the polygon $\overline{P}_{\gamma}$ as the image of $\overline{\mathcal S}_{A\cup 0}$ under the projection $\varphi_{\gamma}.$ 

The Minkowski integral $\overline{I}_{\gamma}=\int\pi_2|_{\overline{P}_{\gamma}}\subset \R$ of the polygon $\overline{P}_{\gamma}$ with respect to the projection $\pi_2\colon\R^2\twoheadrightarrow\bigslant{\R^2}{\langle e_3\rangle}$ is the interval $[0,\area(\overline{P}_{\gamma})].$

Finally, since the Minkowski integration with respect to $\rho_2$ (or $\pi_2$) commutes with the projections $\varphi_{\gamma},$ one can obtain the interval $\overline{I}_{\gamma}\subset\R\langle e_3\rangle$ as the image of the polytope $\overline{\Sigma}_A$ under the projection $\varphi_{\gamma}\colon\R^n\to\R$ defined by $\alpha_i\mapsto \gamma_i\cdot e_3$ (here we also use a similar abuse of notation). Therefore, we have $$[0,\area(\overline{P}_{\gamma})]=\overline{I}_{\gamma}=[\min_{y\in \overline{\Sigma}_A}(\gamma(y)),\max_{y\in \overline{\Sigma}_A}(\gamma(y))].$$

Thus we obtained that the value $\max_{y\in \overline{\Sigma}_A}(\gamma(y))$ of the support function of the polytope $\overline{\Sigma}_A$ on the covector $\gamma$ is exactly $\area \overline{P}_{\gamma},$ which concludes the proof of the proposition.
\end{proof}

\begin{rem}
The polytope $\overline{\Sigma}_A\subset\R^A$ satisfies only one of the two  properties: its support function attains the desired values on covectors with non-negative entries. This property does not uniquely define the polytope. If, however, the polytope is homogeneous, then its support function is uniquely defined by its values on covectors with non-negative coordinates. 
\end{rem}

\begin{proof}[Proof of Lemma \ref{lemma:iterated}]
Denote by $\Delta_{\gamma}$ and $P_{\gamma}$ the images of $\Omega_A$ and $\mathcal S_{A\cup 0}$ under the projections $\varphi_{\gamma}.$ 

The covector $\gamma$ defines a subdivision $\{a_1=w_0<\ldots<w_k=a_n\}$ of the interval $\conv A.$ 

Both polygons $\overline{P}_{\gamma}$ and $P_{\gamma}$ lie below the piecewise linear curve formed by the set of their common edges, which are fiber intervals of the facets $$\conv(\{(w_j,0,\gamma(w_j)),(w_{j+1},0,\gamma(w_{j+1})), (0,1,0)\})\subset\Delta_{\gamma}\cap\overline{\Delta}_{\gamma},$$ 
which implies the following observation.

Consider the projection $p\colon\R^2\to\R$ forgetting the second coordinate, and for each $\xi\in\R\langle e_2\rangle,$ take the intervals $J_{\xi}=p^{-1}(\xi)\cap P_{\gamma}$ and $\overline{J}_{\xi}=p^{-1}(\xi)\cap \overline{P_{\gamma}}.$ One can easily see that we have the inclusion $J_{\xi}\subset\overline{J}_{\xi},$ moreover, for every $\xi\in\R$ such that $J_{\xi}\neq\emptyset,$ we have $$\max_{(\xi,\tau)\in J_{\xi}}\tau=\max_{(\xi,\tau)\in\overline{J}_{\xi}}\tau.$$

The latter implies that the Minkowski interval $I_{\gamma}=\int \pi_2|_{P_{\gamma}}$ is contained in the interval $\overline{I}_{\gamma}=\int \pi_2|_{\overline{P}_{\gamma}}=[0,\area(\overline{P}_{\gamma})]$ and  that the right endpoints of these two intervals coincide. 

Finally, since taking the Minkowski integral with respect to $\pi_2$ (or $\rho_2$) commutes with the projections $\varphi_{\gamma},$ the interval $I_{\gamma}$ is the image of the polytope $\Sigma_A$ under the projection $\varphi_{\gamma}.$ Therefore, we have $$I_{\gamma}=[\min_{y\in \Sigma_A}(\gamma(y)),\max_{y\in \Sigma_A}(\gamma(y))]=[\min_{y\in \Sigma_A}(\gamma(y)),\area(\overline{P}_{\gamma})].$$

Thus we obtained that the value $\max_{y\in \Sigma_A}(\gamma(y))$ of the support function of the polytope $\Sigma_A$ on the covector $\gamma$ is exactly $\area \overline{P}_{\gamma},$ which concludes the proof of the lemma.
\end{proof}
Combining Lemma \ref{lemma:iterated}, Proposition \ref{prop:non_negative} and Theorem \ref{theor:Morse_support}, we obtain formula (\ref{eq:Maxwell_support}) below. 
\begin{theor}\label{theor:np_of_ms}
Up to a global linear summand, the support function ${\mathfrak m}_A$ of the Newton polytope of the Maxwell stratum can be computed via the following formula: 
\begin{equation}\label{eq:Maxwell_support}
2\cdot\mathfrak m_A(\gamma)=\Big[\int\rho_2|_{\int\rho_1|_{\Omega_A}}\Big](\gamma)+\mathscr B_{-\gcd}(\gamma)-4[S_A](\gamma),
\end{equation}
where $[P](\gamma)$ stands for the value of the support function of a polytope $P$ at the covector $\gamma,$ and $S_A$ is the $A$-secondary polytope.
\end{theor}

\vspace{1ex}

{\sc London Institute for Mathematical Sciences} 

aes@lims.ac.uk

\vspace{1ex}

{\sc Department of Mathematical Sciences, University of Copenhagen, Denmark}  

av@math.ku.dk

\begin{thebibliography}{99}
 \bibitem{bbm} B. Bertrand, E. Brugall´e, and G. Mikhalkin, {\it Tropical Open Hurwitz numbers}, Rend.
Semin. Mat. Univ. Padova, 125 (2011) 157--171, arXiv:1005.4628.

	\bibitem{bs0} L. J. Billera, B. Sturmfels, {\it Fiber polytopes,} Ann. of Math., 135 (1992), 527-549.

	\bibitem{bs} L. J. Billera, B. Sturmfels, {\it Iterated fiber polytopes,} Mathematika, 41 (1994), 348-363.

 \bibitem{hannah}  R. Cavalieri, P. Johnson, H. Markwig, {\it Tropical Hurwitz numbers}, J. Algebr. Comb.,
32 (2010) 241--265, arXiv:0804.0579.

\bibitem{edm} J. Edmonds, {\it Submodular functions, matroids, and certain polyhedra}, In  Combinatorial Structures and their Applications, Gordon and
Breach, New York (1970) 69–87.
 
    \bibitem{jems} A. Esterov, {\it Characteristic classes of affine varieties and Plücker formulas for affine morphisms,} J. Eur. Math. Soc. 20 (2018), no. 1, pp. 15–59, arXiv:1305.3234.
	\bibitem{GKZ} I. M. Gelfand, M. M. Kapranov, and A. V. Zelevinsky, {\it Discriminants, resultants, and multidimensional determinants,} Mathematics: Theory \& Applications. Birkhäuser Boston Inc., Boston, MA, 1994.

\bibitem{lz}
S. Lando, D. Zvonkine, {\it Counting ramified coverings and intersection theory on spaces of rational functions I (Cohomology of Hurwitz spaces)}, MMJ, 7 (2007) 85-107, arXiv:math/0303218.

\bibitem{lov} L. Lov\'asz, {\it Submodular functions and convexity}, Mathematical programming: The state of the art, Bonn, 235-257, 1982.

\bibitem{maclagansturmfels}
D. Maclagan, B. Sturmfels, {\it Introduction to Tropical Geometry}, 2015, Graduate Studies in Math. 161.

\bibitem{mikhalkinrau} G. Mikhalkin, J. Rau, {\it Tropical Geometry}, 2018,

\href{https://www.math.uni-tuebingen.de/user/jora/downloads/main.pdf}{https://www.math.uni-tuebingen.de/user/jora/downloads/main.pdf}.

	\bibitem{V23} A. Voorhaar, {\it The Newton Polytope of the Morse Discriminant of a Univariate Polynomial,} Advances in Mathematics,  432(2023), 109275. 
\end{thebibliography}
\end{document}